\documentclass[11pt]{article}
\usepackage{amsmath,amssymb,amsthm}
\usepackage{epsfig}
\usepackage{color}
\usepackage{enumerate}
\usepackage{algorithm}
\usepackage{algorithmic}
\usepackage{hyperref}
\usepackage{enumerate}
\usepackage{graphicx}
\usepackage{tikz}
\usepackage{tkz-graph}
\usepackage{tkz-berge}
\usetikzlibrary{arrows.meta,shapes.arrows}
\hypersetup{colorlinks=true, linkcolor=blue, citecolor=blue,
urlcolor=blue}

\hypersetup{colorlinks=true}

\hypersetup{colorlinks=true, linkcolor=blue, citecolor=blue,urlcolor=blue}

\title{Limited packings: related vertex partitions and duality issues}
\date{}
\author {
Azam Sadat Ahmadi, Nasrin Soltankhah\thanks{Corresponding author}\ \ and Babak Samadi\vspace{2mm}\\
Department of Mathematics, Faculty of Mathematical Sciences, Alzahra University,\\ Tehran, Iran\vspace{1mm}\\
{\tt as.ahmadi@alzahra.ac.ir}\\
{\tt soltan@alzahra.ac.ir}\\
{\tt b.samadi@alzahra.ac.ir}
}
\date{}

\newtheorem{theorem}{Theorem}[section]

\newtheorem{lemma}[theorem]{Lemma}

\theoremstyle{definition}

\theoremstyle{remark}

\begin{document}

\maketitle

\begin{abstract}
A $k$-limited packing partition ($k$LP partition) of a graph $G$ is a partition of $V(G)$ into $k$-limited packing sets. We consider the $k$LP partitions with minimum cardinality (with emphasis on $k=2$). The minimum cardinality is called $k$LP partition number of $G$ and denoted by $\chi_{\times k}(G)$. This problem is the dual problem of $k$-tuple domatic partitioning as well as a generalization of the well-studied $2$-distance coloring problem in graphs.

We give the exact value of $\chi_{\times2}$ for trees and bound it for general graphs. A section of this paper is devoted to the dual of this problem, where we give a solution to an open problem posed in $1998$. We also revisit the total limited packing number in this paper and prove that the problem of computing this parameter is NP-hard even for some special families of graphs. We give some inequalities concerning this parameter and discuss the difference between $2$TLP number and $2$LP number with emphasis on trees.
\end{abstract}
\vspace{0.5mm}


\section{Introduction and preliminaries} 

Throughout this paper, we consider $G$ as a finite simple graph with vertex set $V(G)$ and edge set $E(G)$. The {\em open neighborhood} of a vertex $v$ is denoted by $N_{G}(v)$, and its {\em closed neighborhood} is $N_{G}[v]=N_{G}(v)\cup \{v\}$ (we simply write $N(v)$ and $N[v]$ if there is no ambiguity). The {\em minimum} and {\em maximum degrees} of $G$ are denoted by $\delta(G)$ and $\Delta(G)$, respectively. For $A,B\subseteq V(G)$, by $[A,B]$ we mean the set of edges with one end point in $A$ and the other in $B$. The subgraph of $G$ induced by $S\subseteq V(G)$ is denoted by $G[S]$. We use \cite{West} as a reference for terminology and notation which are not explicitly defined here.

Following \cite{hh1}, a vertex subset $S$ of a graph $G$ with $\delta(G)\geq k-1$ is a \textit{$k$-tuple dominating set} ($k$TD set) in $G$ if $|N[v]\cap S|\geq k$ for each $V\in V(G)$. The \textit{$k$-tuple domination number} $\gamma_{\times k}(G)$ is the smallest number of vertices in a $k$TD set in $G$. A \textit{$k$-tuple domatic partition} ($k$TD partition) of graph $G$ is a vertex partition of $G$ into $k$TD sets. The maximum number of sets in a vertex partition of $G$ into $k$TD sets is called \textit{$k$-tuple domatic number} and denoted by $d_{\times k}(G)$ (the study of this concept was initiated by Harary and Haynes \cite{hh2} in $1998$, and further investigated in many papers, for instance \cite{ssm,v}). Note that when $k=1$, $S$ and $\gamma_{\times1}(G)$ are the usual dominating set and domination number $\gamma(G)$, respectively. Moreover, $d_{\times1}(G)=d(G)$ is the well-studied domatic number (see \cite{ch}). The reader can find comprehensive information and overviews of domination theory in the books \cite{hhh,hhs}. 

A set of vertices $B\subseteq V(G)$ is called a \textit{packing} (resp. an \textit{open packing}) in $G$ provided that $N[u]\cap N[v]=\emptyset$ (resp. $N(u)\cap N(v)=\emptyset$) for each distinct vertices $u,v\in V(G)$. The maximum cardinality of a packing (resp. open packing) is called the \textit{packing number} (resp. \textit{open packing number}), denoted $\rho(G)$ (resp. $\rho_{o}(G)$). For more information about these topics, the reader can consult \cite{hhh} and \cite{hhs}. In $2010$, Gallant et al. (\cite{gghr}) introduced the concept of limited packing in graphs. In fact, a set $B\subseteq V(G)$ is said to be a \textit{$k$-limited packing} ($k$LP) in the graph $G$ if $|B\cap N[v]|\leq k$ for each vertex $v$ of $G$. The \textit{$k$-limited packing number} $L_{k}(G)$ is the maximum cardinality of a $k$LP in $G$. They also exhibited some real-world applications of it in network security, market situation, NIMBY and codes. This concept was next investigated in many papers, for instance, \cite{bcl,gz,s}. Similarly, a set $B\subseteq V(G)$ is said to be a \textit{$k$-total limited packing} ($k$TLP) if $|B\cap N(v)|\leq k$ for each vertex $v$ of $G$. The \textit{$k$-total limited packing number} $L_{k,t}(G)$ is the maximum cardinality of a $k$TLP in $G$. This concept was first studied in \cite{hms} and some theoretical applications of it were given in \cite{hmsv}. It is easy to see that the latter two concepts are the same with the concepts of packing and open packing when $k=1$.

The concept of distance coloring of graphs was initiated by Kramer and Kramer \cite{kk2,kk1} in $1969$. A {\em $2$-distance coloring} of a graph $G$ is a mapping of $V(G)$ to a set of colors such that any two vertices at distance at most two receive different colors. The minimum number of colors $k$ for which there is a $2$-distance coloring of $G$ is the {\em $2$-distance chromatic number} $\chi_{2}(G)$. We observe that the color classes of a $2$-distance coloring are the very packing sets ($1$LP sets). In view of this, a generalization of this concept is the vertex partitioning of graphs into $k$LP sets. A $k$-\textit{limited packing partition} (\textit{$k$LP partition}) is a vertex partition into $k$LP sets, and the $k$-limited packing partition number \textit{$k$LP partition number} $\chi_{\times k}(G)$ is the minimum cardinality taken over all $k$LP partitions of $G$. This problem can also be considered as the dual of $k$TD partition problem. We investigate and revisit these kinds of vertex partitions in this paper. Regarding $k$LP sets, our main emphasis is on $k=2$. One reason is that, in such problems, we lose some important families of graphs for large values of $k$ (for example, $\gamma_{\times k}$ and $d_{\times k}$ cannot be defined for trees when $k\geq3$) or we have a trivial problem (for example, $L_{\times k}(G)=|V(G)|$ and $\chi_{\times k}(G)=1$ if $k\geq \Delta(G)+1$). On the other hand, many results for $k\in \{1,2\}$ can be generalized to the general case $k$. Moreover, one may obtain stronger results when dealing with small values of $k$. 

This paper is organized as follows. We solve a characterization problem concerning the $k$TD partitions in graphs, leading to a solution to an open problem posed in $1998$ by Harary and Haynes (\cite{hh2}). In Section $3$, we give the exact values of $2$LP partition number for all trees, whose proof employs an efficient algorithm in order to obtain an optimal $k$LP partition. Several sharp inequalities concerning this parameter are given. In particular, we bounds $L_{2}$ and $\chi_{\times2}$ for the lexicographic product graphs from below and above. In Section $4$, we prove that the decision problem associated with $L_{2,t}$ is NP-complete even when restricted to chordal graphs and bipartite graphs. Some sharp inequalities concerning this parameter are given in this section. Finally, we discuss the relationship between $L_{2,t}$ and $L_{2}$ by bounding their difference for trees just in terms of the order.

For the sake of convenience, for any graph $G$ by an $\eta(G)$-set with $\eta\in \{L_{2},\chi_{\times k},\gamma_{\times2},\rho,\rho_{o},L_{2,t}\}$ we mean a $2$LP set, $k$LP partition, $2$TD set, packing set, open packing set and $2$TLP set in $G$ of cardinality $\eta(G)$, respectively.


\section{Duality issues}

Suppose that $\mathbb{P}=\{P_{1},\cdots,P_{\chi_{\times k}(G)}\}$ is a $\chi_{\times k}(G)$-partition of an arbitrary graph $G$. Let $v$ be a vertex of maximum degree of $G$. By definition, we deduce that $\Delta(G)+1=\sum_{i=1}^{\chi_{\times k}(G)}|N[v]\cap P_{i}|\leq k\chi_{\times k}(G)$. So, we have the following simple but important inequality
\begin{equation}\label{Delta}
\chi_{\times k}(G)\geq\lceil\frac{\Delta(G)+1}{k}\rceil.
\end{equation}
We shall make use of (\ref{Delta}), for $k=2$, in many places of this paper.

Recall that a maximization problem \textbf{M} and a minimization problem \textbf{N}, defined on the same instances (such as graphs or digraphs), are dual problems if the value of every candidate solution \textit{M} to \textbf{M} is less than or equal to the value of every candidate solution \textit{N} to \textbf{N}. Usually, the ``value" is cardinality.

Harary and Haynes (\cite{hh2}) in $1998$ proved that $d_{\times k}(G)\leq\lfloor(\delta(G)+1)/k\rfloor$ for each graph $G$ with $\delta(G)\geq k-1$. With this in mind, the inequality (\ref{Delta}) shows that the $k$LP partition problem and the $k$TD partition problem are dual. They also posed the following open problem:\vspace{1.5mm}\\
\textbf{Problem A.} (\cite{hh2}) Characterize all graphs $G$ with $\delta(G)\geq k-1$, for which $d_{\times k}(G)=\lfloor\frac{\delta(G)+1}{k}\rfloor$.\vspace{1mm}

We completely solve the above problem in this section from the structural points of view. To do so, let $q$ be a positive integer. By the division algorithm, there exist unique integers $t$ and $r$ satisfying $q+1=tk+r$ and $0\leq r\leq k-1$. We now begin with $t$ arbitrary graphs $H_{1},\cdots,H_{t}$ with $\delta(H_{i})\geq k-1$ for each $1\leq i\leq t$, and $\delta(H_{i'})\leq2k-2$ for some $1\leq i'\leq t$. We next add some edges with one end point in $V(H_{i})$ and the other in $V(H_{j})$ in such a way that\vspace{0.5mm}\\
(i) the minimum degree of the resulting graph equals $q$, and\vspace{0.5mm}\\
(ii) every vertex of $H_{i}$ has at least $k$ neighbors in $H_{j}$ for each $1\leq i\neq j\leq t$.\vspace{0.5mm}

Let $\Omega$ be the family of the above-constructed graphs $G$. Note that the condition ``$\delta(H_{i'})\leq2k-2$ for some $1\leq i'\leq t$" could not be removed. Otherwise, the statement (ii) would imply that $\delta(G)\geq kt+r=q+1$, which contradicts the statement (i). Figure \ref{G} depicts a representative member of $\Omega$.

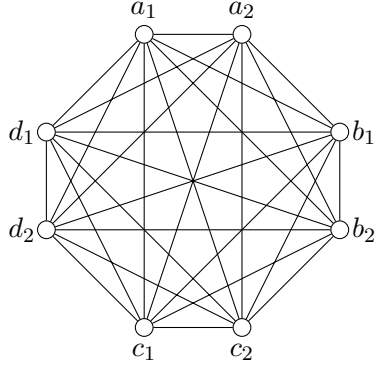
\begin{figure}[h]
\centering
\begin{tikzpicture}[scale=.65, transform shape]
\node [draw, shape=circle] (a1) at (0,0) {};
\node [draw, shape=circle] (a2) at (2,0) {};

\node [draw, shape=circle] (c1) at (0,-6) {};
\node [draw, shape=circle] (c2) at (2,-6) {};

\node [draw, shape=circle] (d1) at (-2,-2) {};
\node [draw, shape=circle] (d2) at (-2,-4) {};

\node [draw, shape=circle] (b1) at (4,-2) {};
\node [draw, shape=circle] (b2) at (4,-4) {};

\node [scale=1.5] at (0,0.5) {$a_1$};
\node [scale=1.5] at (2,0.5) {$a_2$};

\node [scale=1.5] at (0,-6.5) {$c_1$};
\node [scale=1.5] at (2,-6.5) {$c_2$};

\node [scale=1.5] at (-2.5,-2) {$d_1$};
\node [scale=1.5] at (-2.5,-4) {$d_2$};

\node [scale=1.5] at (4.5,-2) {$b_1$};
\node [scale=1.5] at (4.5,-4) {$b_2$};

\draw[](a1)--(a2);
\draw[](b1)--(b2);
\draw[](c1)--(c2);
\draw[](d1)--(d2);
\draw(a1)--(d1)--(a2)--(d2)--(a1)--(b2)--(a2)--(b1)--(a1);
\draw(c1)--(d1)--(c2)--(d2)--(c1)--(b2)--(c2)--(b1)--(c1);
\draw(c1)--(a1)--(c2)--(a2)--(c1);
\draw(b1)--(d1)--(b2)--(d2)--(b1);

\end{tikzpicture}
\caption{A graph $H\in \Omega$ with $q=7$, $k=2$ and $t=4$.}\label{G}
\end{figure} 

\begin{theorem}\label{Char}
Let $G$ be a graph with minimum degree at least $k-1$. Then, $d_{\times k}(G)=\lfloor\frac{\delta(G)+1}{k}\rfloor$ if and only if $G\in \Omega$.
\end{theorem}  
\begin{proof}
Suppose first that $d_{\times k}(G)=\lfloor(\delta(G)+1)/k\rfloor$ for a graph $G$ with $\delta(G)\geq k-1$. Let $v$ be a vertex of minimum degree in $G$. Hence, we can write $\delta(G)+1=\deg(v)+1=t'k+r'$, where $0\leq r'\leq k-1$ and $t'=d_{\times k}(G)=\lfloor(\delta(G)+1)/k\rfloor$. Let $\mathbb{P}=\{S_{1},\cdots,S_{d_{\times k}(G)}\}$ be a $d_{\times k}(G)$-partition. By definition, we have that $\delta(G[S_{i}])\geq k-1$ for each $1\leq i\leq d_{\times k}(G)$, and that every vertex in $S_{i}$ has at least $k$ neighbors in $S_{j}$ for each $1\leq i\neq j\leq d_{\times k}(G)$. If $\delta(G[S_{i}])\geq2k-1$ for each $1\leq i\leq d_{\times k}(G)$, we have in particular that 
$$\delta(G)+1=\deg(v)+1\geq2k-1+k(t'-1)+1>t'k+r',$$
a contradiction. Therefore, $\delta(G[S_{i'}])\leq2k-2$ for some $1\leq i'\leq d_{\times k}(G)$.

It is now readily checked that $t'=d_{\times k}(G)$, $G[S_{1}],\cdots,G[S_{d_{\times k}(G)}]$, $r'$ and $\delta(G)$ correspond to $t$, $H_{1},\cdots,H_{t}$, $r$ and $q$, respectively, in the description of the members of $\Omega$. Thus, $G\in \Omega$.

Conversely, suppose that $G\in \Omega$. It is then easy to see that $\mathbb{P}=\{V(H_{1}),\cdots,V(H_{t})\}$ is a $k$TD partition of $G$. In particular, this shows that 
$$\lfloor\frac{\delta(G)+1}{k}\rfloor=\lfloor\frac{q+1}{k}\rfloor=\lfloor\frac{tk+r}{k}\rfloor=t\leq d_{\times k}(G)\leq \lfloor\frac{\delta(G)+1}{k}\rfloor.$$
This implies the desired equality. In fact, we have proved that $d_{\times k}(G)=\lfloor\frac{\delta(G)+1}{k}\rfloor$ if and only if $G\in \Omega$.
\end{proof}


\section{Vertex partitioning into $2$-limited packing sets}

In the following theorem, we prove that the lower bound in (\ref{Delta}), for $k=2$, gives the exact value of the $2$LP partition number of all trees. Moreover, its proof gives us a practical method to obtain an optimal $2$ partition of a given tree into $2$LP sets.

\begin{theorem}\label{Tree}
For any tree $T$, $\chi_{\times2}(T)=\lceil\frac{\Delta(T)+1}{2}\rceil$.
\end{theorem}
\begin{proof}
The equality trivially holds when $|V(T)|=1$. So, we may assume that $|V(T)|\geq2$. By taking the inequality (\ref{Delta}) with $k=2$ into account, it suffices to construct a $2$LP partition of $T$ of cardinality $p=\lceil(\Delta(T)+1)/2\rceil$.

Let $x$ be vertex of maximum degree of $T$. We root $T$ at $x$. We distinguish two cases depending on the behavior of $\Delta(T)$.\vspace{0.5mm}\\
\textit{Case 1.} $\Delta(T)\equiv0$ (mod $2$). Let $N(x)=\{u_{1},\cdots,u_{\Delta(T)}\}$. We define $f:V(T)\rightarrow \mathbb{N}$ as follows. Let $f(x)=p$ and $f(u_{2i-1})=f(u_{2i})=i$ for each $1\leq i\leq \Delta(T)/2$. If $T$ is a star, then we stop the process. Otherwise, there exists a child $u_{k}$ of $x$, for some $1\leq k\leq \Delta(T)$, which is not a leaf. Since $\deg(u_{k})\leq \deg(x)=\Delta(T)$, it follows that $u_k$ has at most $\Delta(T)-1$ children, say $w_{1},\cdots,w_{t}$. If $t$ is odd, let $f(w_t)=p$ and $f(w_{2i-1})=f(w_{2i})=c_i$, in which $c_{i}\in\{1,\cdots,\Delta(T)/2\}\setminus \{f(u_k)\}$ for each $1\leq i\leq(t-1)/2$. If $t$ is even, let $f$ assign the first $t/2$ labels from $\{1,\cdots,\Delta(T)/2\}\setminus\{f(u_{k})\}$ to the vertices $w_{1},\cdots,w_{t}$ in such a way that $f(w_{2i-1})=f(w_{2i})$ for each $1\leq i\leq t/2$. This process is continued until all descendants of $x$ are labeled under $f$.\vspace{0.5mm}\\
\textit{Case 2.} $\Delta(T)\equiv1$ (mod $2$). We define $f:V(T)\rightarrow \mathbb{N}$ as follows. Let $f(x)=f(u_{\Delta(T)})=p$ and $f(u_{2i-1})=f(u_{2i})=i$ for every $1\leq i\leq(\Delta(T)-1)/2$. If $T$ is isomorphic to a star, we then stop the process. Otherwise, $x$ has a child $u_r$ which is not a leaf. Let $z_{1},\cdots,z_{s}$ be the children of $u_r$. Notice that $s\leq \Delta(T)-1$. We now consider two more possibilities.\vspace{0.5mm}

\textit{Subcase 2.1.} $r=\Delta(T)$. Let $f$ assign $\lceil s/2\rceil$ labels from $\{1,\cdots,(\Delta(T)-1)/2\}$ to the vertices $z_{1},\cdots,z_{s}$ so that each labels appears on at most two vertices.\vspace{0.5mm}

\textit{Subcase 2.2.} $1\leq r\leq \Delta(T)-1$. If $s$ is odd, then $s\leq \Delta(T)-2$ (since $s$ and $\Delta(T)-1$ have different parity). In such a situation, $f$ assigns the first $(s-1)/2$ labels from $\{1,\cdots,(\Delta(T)-1)/2\}\setminus \{f(u_{r})\}$ to the vertices $z_{1},\cdots,z_{s-1}$ in such a way that each label appears on precisely two vertices, and assigns $p$ to $z_s$. If $s$ is even, let $(f(z_{s-1}),f(z_{s}))=(f(u_{r}),p)$ and let $f$ assign $(s-2)/2$ labels from $\{1,\cdots,(\Delta(T)-1)/2\}\setminus \{f(u_{r})\}$ to the vertices $z_{1},\cdots,z_{s-2}$ such that each label appears on exactly two vertices. Again, this process ends when all descendants of $x$ are labeled under $f$.\vspace{0.5mm}

In each case, we observe that $\big{\{}\{v\in V(T)\mid f(v)=i\}\mid 1\leq i\leq p\big{\}}$ is a $2$LP partition of $T$ of cardinality $p=\lceil(\Delta(T)+1)/2\rceil$, as desired.
\end{proof}

We define the family $\Lambda$ as follows. Let $H$ be an $r$-partite graph with partite sets $V_1,V_2,\cdots,V_r$ such that\vspace{0.5mm}\\
$\bullet$ $|V_1|,|V_2|,\cdots,|V_r|\equiv0$ (mod $2$) and\vspace{0.5mm}\\
$\bullet$ $H[V_i\cup V_j]$ is a 2-regular graph for each $1\leq i\neq j\leq r$.\vspace{0.5mm}\\
Let $G$ be obtained from $H$ by constructing a prefect matching using the vertices of $V_i$ for each $1\leq i\leq r$.
Let $\Lambda$ be the family of all graphs $G$. A representative member of $\Lambda$ is depicted in Figure \ref{Fig2}.

\begin{figure}[h]
\centering
\begin{tikzpicture}[scale=.4, transform shape]
\node [draw, shape=circle] (a1) at (0,0) {};
\node [draw, shape=circle] (a2) at (2,0) {};
\node [draw, shape=circle] (a3) at (4,0) {};
\node [draw, shape=circle] (a4) at (6,0) {};

\node [draw, shape=circle] (b1) at (9,-5.2) {};
\node [draw, shape=circle] (b2) at (8,-6.94) {};
\node [draw, shape=circle] (b3) at (7,-8.67) {};
\node [draw, shape=circle] (b4) at (6,-10.4) {};

\node [draw, shape=circle] (c1) at (-3,-5.2) {};
\node [draw, shape=circle] (c2) at (-2,-6.94) {};
\node [draw, shape=circle] (c3) at (-1,-8.67) {};
\node [draw, shape=circle] (c4) at (0,-10.4) {};

\node [scale=1.5] at (0,0.7) {a1};
\node [scale=1.5] at (2,0.7) {a2};
\node [scale=1.5] at (4,0.7) {a3};
\node [scale=1.5] at (6,0.7) {a4};

\node [scale=1.5] at (9.7,-5.2) {b1};
\node [scale=1.5] at (8.7,-6.94) {b2};
\node [scale=1.5] at (7.7,-8.67) {b3};
\node [scale=1.5] at (6.7,-10.4) {b4};

\node [scale=1.5] at (-3.7,-5.2) {c1};
\node [scale=1.5] at (-2.7,-6.94) {c2};
\node [scale=1.5] at (-1.7,-8.67) {c3};
\node [scale=1.5] at (-0.7,-10.4) {c4};

\draw(a1)--(a2);
\draw(a3)--(a4);
\draw(b1)--(b2);
\draw(b3)--(b4);
\draw(c1)--(c2);
\draw(c3)--(c4);
\draw(c1)--(a1)--(c2)--(b1)--(c1)--(a2)--(c2)--(b2)--(c1);
\draw(c3)--(a3)--(c4)--(a4)--(c3)--(b3)--(c4)--(b4)--(c3);
\draw(a4)--(b1)--(a3)--(b2)--(a4);
\draw(a1)--(b4)--(a2)--(b3)--(a1);

\end{tikzpicture}
\caption{A graph $G\in \Lambda$ with $r=3$ and $|V_i|=4$ for each $1\leq i \leq 3$.}\label{Fig2}
\end{figure}

\begin{theorem}
For any graph $G$ of order $n$ and size $m$,
$$\chi_{\times 2}(G)\geq \frac{1}{2}\Big{(}1+max\Big{\{}\frac{2m}{n},\sqrt{1+\frac{4m-2n}{L_2(G)}}\Big{\}}\Big{)}$$
with equality if and only if $G \in \Lambda$.
\end{theorem}
\begin{proof}
Let $\{P_1,P_2,\cdots,P_{\chi_{\times2}(G)}\}$ be a $2$LP partition of $G$. Then,
\begin{equation}\label{Ine1}
\sum_{v\in V(G)}\sum _{i=1}^{\chi_{\times2}(G)}|N_G[v]\cap P_i|\leq \sum_{v\in V(G)}\sum_{i=1}^{\chi_{\times2}(G)}2=2n\chi_{\times2}(G).
\end{equation}

On the other hand,
\begin{equation*}
\begin{array}{lcl}
\sum_{v\in V(G)}\sum_{i=1}^{\chi_{\times2}(G)}|N_G[v]\cap P_i|&=&\sum_{v\in V(G)}|N_G[v]\cap (\cup_{i=1}^{\chi_{\times 2}(G)}P_i)|\\
&=&\sum_{v\in V(G)}|N_G[v]|=\sum_{v\in V(G)}(\deg(v)+1)=2m+n.
\end{array}
\end{equation*}
So, $2m+n\leq 2n\chi_{\times2}(G)$. Therefore,
\begin{equation}\label{LB1}
\chi_{\times2}(G)\geq m/n+1/2.
\end{equation}

We now characterize all graphs for which the equality holds in the lower bound (\ref{LB1}). Suppose that $\chi_{\times 2}(G)=m/n+1/2$ for a graph $G$. Therefore, the inequality (\ref{Ine1}) necessarily holds with equality. This means that $|N_G[v]\cap P_i|=2$ for each $v\in V(G)$ and $1\leq i\leq \chi_{\times 2}(G)$. Hence, each vertex of $P_i$ is adjacent to precisely one vertex in $P_i$ and that it has exactly two neighbors in $P_j$ for every $1\leq j\neq i\leq \chi_{\times2}(G)$. This means that the induced subgraph $G[P_i]$ is $1$-regular, and hence $|P_i|$ is even for every $1\leq i\leq \chi_{\times2}(G)$. Since every vertex of $P_i$ is adjacent to precisely two vertices of $P_j$ for each $1\leq i\neq j\leq \chi_{\times2}(G)$, we infer that $|P_i|=|P_j|$. We now observe that $G\in \Lambda$ because $P_i$ and $\chi_{\times2}(G)$ correspond to $V_i$ and $r$, respectively, given in the definition of $\Lambda$.

Conversely, let $G\in \Lambda$. Since every $V_i$ in $G$ is a $2$LP set, it follows that $\chi_{\times 2}(G)\leq r$.
On the other hand, $\chi_{\times2}(G)\geq(1+2m/n)/2=(1+n(2r-1)/n)/2=r$. Thus, $\chi_{\times2}(G)=m/n+1/2$.

Let $\{P_1,P_2,\cdots,P_{\chi_{\times2}(G)}\}$ be a $\chi_{\times2}(G)$-partition with $|P_1|\leq|P_2|\leq \cdots \leq|P_{\chi_{\times2}(G)}|$. Then,
\begin{equation}\label{Ine2}
\begin{array}{lcl}
m&=&\sum_{i=1}^{\chi_{\times2}(G)}|[P_i,P_i]|+\sum_{1\leq i<j\leq\chi_{\times2}(G)}|[P_i,P_j]|\\
&\leq&\sum_{i=1}^{\chi_{\times 2}(G)}\frac{|P_i|}{2}+2|P_1|(\chi_{\times2}(G)-1)+2|P_2|(\chi_{\times2}(G)-2)+\cdots+2|P_{\chi_{\times2}(G)}|(\chi_{\times2}(G)-\chi_{\times2}(G))\\
&=&\sum_{i=1}^{\chi_{\times2}(G)}\frac{|P_i|}{2}+\sum_{i=1}^{\chi_{\times2}(G)}2|P_i|(\chi_{\times2}(G)-i)\\
&\leq& \frac{n}{2}+L_2(G)\chi_{\times2}(G)(\chi_{\times2}(G)-1).
\end{array}
\end{equation}
So, $2L_2(G)\chi_{\times2}(G)^{2}-2L_2(G)\chi_{\times2}(G)+n-2m\geq0$. Solving this equation for $\chi_{\times 2}(G)$, we get
\begin{equation}\label{LB2}
\chi_{\times2}(G)\geq\big{(}1+\sqrt{1+(4m-2n)/L_2(G)}\big{)}/2.
\end{equation}
Suppose that the equality holds in (\ref{LB2}). In such a situation, all inequalities in (\ref{Ine2}) hold with equality. In fact, we have\vspace{0.75mm}\\
$(i)$ $|[P_i,P_i]|=|P_i|/2$ for each $1\leq i\leq \chi_{\times 2}(G)$,\vspace{0.75mm}\\ 
$(ii)$ $\sum_{1\leq i<j\leq\chi_{\times2}(G)}|[P_i,P_j]|=\sum_{i=1}^{\chi_{\times2}(G)}2|P_i|(\chi_{\times2}(G)-i)$, and\vspace{0.75mm}\\
$(iii)$ $|P_i|=L_2(G)$ for each $1\leq i\leq \chi_{\times2}(G)$.\vspace{0.75mm}

Since $|[P_i,P_i]|=|P_i|/2$ and $P_i$ is a $2$LP set, each vertex of $P_i$ has exactly one neighbor in $P_i$. This implies that the induced subgraph $G[P_i]$ is $1$-regular for each $1\leq i\leq \chi_{\times2}(G)$. We deduce from $(ii)$ that every vertex $v\in P_i$ has precisely two neighbors in $P_j$ and hence $2|P_i|=|[P_i,P_j]|=2|P_j|$ for each $1\leq i<j\leq \chi_{\times2}(G)$. Particularly, all $|P_i|$'s are even. We now observe that $P_i$ and $\chi_{\times 2}(G)$ correspond to $V_i$ and $r$, respectively, given in the description of the members of $\Lambda$. Thus, $G\in \Lambda$.

Conversely, let $G\in \Lambda$. So, $\chi_{\times2}(G)\leq r$. Since $V_i$ is a $2$LP set for each $1\leq i\leq r$, it follows that $L_2(G)\geq n/r$. On the other hand, let $S$ be a $\gamma_{\times 2}(G)$-set (such a set exists because $\delta(G)\geq1$) and let $B$ be a $L_2(G)$-set. We set $Q=\{(x,y)|x\in B,y\in S,x\in N[y]\}$. By the definitions, the closed neighborhood of every vertex in $B$ (resp. $S$) contains at least (resp. at most) two vertices in $S$ (resp. $B$). This shows that $2|B|\leq|Q|\leq2|S|$. Hence, $L_2(G)\leq \gamma_{\times2}(G)$. Since $V_i$ is also a $2$-tuple dominating set for each $1\leq i\leq r$, we have $\gamma_{\times2}(G)\leq n/r$. Therefore, $L_2(G)=n/r$. Moreover, by using $4m=2n(2r-1)$, we have $\chi_{\times2}(G)\geq \frac{1}{2}\big{(}1+\sqrt{1+(4m-2n)/L_2(G)}\big{)}=r$. This leads to the equality in the lower bound (\ref{LB2}).

All in all, the desired lower bound follows from (\ref{LB1}) and (\ref{LB2}). Furthermore, $\Lambda$ is the family of all graphs attaining the bound.
\end{proof}

There is a simple but strong relationship between the $2$LP partition number and the well-studied $2$-distance chromatic number of graphs stated as follows.

\begin{theorem}
For any graph $G$, $\chi_{\times 2}(G)\leq \lceil\frac{\chi_2(G)}{2}\rceil$.
\end{theorem}
\begin{proof}
Let $f$ be a $2$-distance coloring with minimum number of colors. Note that the color classes $B_i$, $1\leq i\leq \chi_2(G)$, give us a vertex partition of $G$ into packing sets. We now construct a $2$LP partition of $G$ as follows. It is obvious that $|N[v]\cap(B_i\cup B_j)|\leq2$ for each $v\in V(G)$ and $1\leq i \neq j \leq \chi_2(G)$. So, it suffices to mutually combine the members of $\mathbb{B}$ to get a $2$LP partition of $G$. This leads to $\chi_{\times 2}(G)\leq \lceil \chi_2(G)/2\rceil$.

It was showed in \cite{nk} that $\chi_2(T)=\Delta(T)+1$ for every tree $T$. Hence, the upper bound is sharp for all trees in view of Theorem \ref{Tree} and the inequality (\ref{Delta}) with $k=2$.
\end{proof}

Nordhaus and Gaddum \cite{NG} exhibited lower and upper bounds on the sum and product of the chromatic number of a graph and its complement in terms of the order. Since then, bounds on $\Phi(G)+\Phi(\overline{G})$ and $\Phi(G)\Phi(\overline{G})$, in which $\Phi$ is any graph parameter, are called Nordhaus-Gaddum inequalities. The reader can consult \cite{ah} for more information on this subject. Regarding the $2$LP numbers, it was proved in \cite{s} that
\begin{equation}\label{B}
L_{2}(G)+L_{2}(\overline{G})\leq n+2
\end{equation}
for each graph of order $n$. We here give such an inequality concerning the $2$LP partition numbers.

\begin{theorem}
For any graph $G$ of order $n$, $\chi_{\times2}(G)+\chi_{\times2}(\overline G)\geq \frac{n+2}{2}$. Moreover, this bound is sharp.
\end{theorem}
\begin{proof}
Let $\mathbb{P}=\{P_1,P_2,\cdots,P_{\chi_{\times2}(G)}\}$ be a $\chi_{\times2}(G)$-partition. Taking the inequality (\ref{Delta}) with $k=2$ into account, we get
\begin{equation}\label{NoGa}
\chi_{\times 2}(G)+\chi_{\times 2}(\overline G)\geq \frac{\Delta(G)+1}{2}+\frac{\Delta(\overline G)+1}{2}=\frac{n+\Delta(G)-\delta(G)+1}{2}\geq \frac{n+1}{2}.
\end{equation}

Suppose that we have $\chi_{\times2}(G)+\chi_{\times2}(\overline G)=(n+1)/2$ for some graph $G$ of order $n$. Taking this into consideration and (\ref{NoGa}), we deduce that $\chi_{\times 2}(G)=(\Delta(G)+1)/2$ and that $G$ is a regular graph, necessarily. Consider an arbitrary $2$LP set $P_i\in \mathbb{P}$ and let $v\in P_i$. Since $\chi_{\times 2}(G)=(\deg(v)+1)/2$, it follows that $v$ necessarily has exactly one neighbor in $P_i$ (and exactly two neighbors in $P_j$ for each $1\leq i\neq j\leq \chi_{\times2}(G)$). In particular, this implies that the subgraph induced by $P_i$ has a perfect matching for each $1\leq i\leq \chi_{\times 2}(G)$. Therefore, $n$ is even. So, we have $\chi_{\times2}(G)+\chi_{\times2}(\overline G)\geq(n+2)/2$.

Let $H$ be obtained from the complete graph $K_{2p}$, for $p\geq3$, by removing a perfect matching. It is clear that $\chi_{\times2}(\overline H)=1$. On the other hand, we observe that every three vertices of $H$ belong to the closed neighborhood of some vertex in $H$. Therefore, $L_{2}(H)=2$. This in particular shows that $\chi_{\times2}(H)=p$. In fact, we have $\chi_{\times2}(H)+\chi_{\times2}(\overline H)=p+1=(|V(H)|+2)/2$. So, the bound is sharp.
\end{proof}

\subsection{$2$-limited packings in lexicographic product graphs}

We recall that the \textit{lexicographic product graph} $G\circ H$ of two graphs $G$ and $H$ is a graph with vertex set $V(G)\times V(H)$ in which two vertices $(g,h)$ and $(g',h')$ are adjacent if either ``$gg'\in E(G)$" or ``$g=g'$ and $hh'\in E(H)$" (see \cite{ImKl} for a comprehensive overview of graph products).

\begin{theorem}\label{lex1}
Let $G$ be a connected graph of order $n\geq2$ and $H$ be an arbitrary graph of order at least two. Then,
$$2\rho(G)\leq L_{2}(G\circ H)\leq2\big{(}n-\Delta(G)\big{)}.$$
\end{theorem}
\begin{proof}
Let $B$ be an $L_{2}(G\circ H)$-set. Consider an arbitrary vertex $v\in V(G)$. Since $G$ has no isolated vertices, there is a vertex $u$ adjacent to $v$. Because $B$ is a $2$LP set in $G\circ H$ and since the vertex $(u,h)$ is adjacent to all vertices in $\{v\}\times V(H)$ for any vertex $h\in V(H)$, it follows that $|B\cap\big{(}\{v\}\times V(H)\big{)}|\leq2$. 

On the other hand, let $w$ be a vertex of maximum degree in $G$ and set $S=N_{G}[w]\times V(H)$. Suppose that $|B\cap S|\geq3$ and that $(g_1,h_1),(g_2,h_2),(g_3,h_3)\in B\cap S$. If $g_1=g_2=g_3=w$, then $(g_1,h_1),(g_2,h_2),(g_3,h_3)\in N_{G\circ H}\big{(}(v,h)\big{)}\cap S$ for each $v\in N_{G}(w)$ and $h\in V(H)$. If only one of $g_1,g_2,g_3$, say $g_1$, is different from $w$, then $(g_1,h_1),(g_2,h_2),(g_3,h_3)\in N_{G\circ H}[(g_1,h_1)]\cap B$. If exactly two of $g_1,g_2,g_3$, say $g_1$ and $g_2$, are different from $w$, we have $(g_1,h_1),(g_2,h_2),(g_3,h_3)\in N_{G\circ H}[(g_3,h_3)]\cap B$. Finally, if $w\neq g_1,g_2,g_3$, then $(g_1,h_1),(g_2,h_2),(g_3,h_3)\in N_{G\circ H}\big{(}(w,h_1)\big{)}\cap B$. In each case, we have a contradiction with the fact that $B$ is a $2$LP set in $G\circ H$. Therefore, $|B\cap S|\leq2$. The above arguments guarantee that
$$L_{2}(G\circ H)=|B|=|B\cap S|+|B\cap(V(G\circ H)\setminus S)|\leq2+2(n-\Delta(G)-1)=2\big{(}n-\Delta(G)\big{)}.$$

Suppose now that $P$ is a $\rho(G)$-set and let $h_1,h_2\in V(H)$. We set $B=P\times \{h_1,h_2\}$. Since the distsnce between any two vertices of $P$ in $G$ is at least three and because $B$ contains two vertices from $\{v\}\times V(H)$ for each $v\in P$, we deduce that $|N_{G\circ H}[(v,h)]\cap B|\leq2$ for each $v\in P$ and $h\in V(H)$. Suppose that $(v_1,h_1),(v_2,h_2),(v_3,h_3)\in N_{G\circ H}\big{(}(x,h)\big{)}\cap B$ for some $x\in V(G)\setminus P$ and $h\in V(H)$. By taking the adjacency rules of lexicographic graph products and the definitions of $P$ and $B$, we may assume that $v_1\neq v_2$ and that $x$ is adjacent to both $v_1$ and $v_2$. Hence, $d_{G}(v_1,v_2)\leq2$. This contradicts the fact that $P$ is a packing set in $G$. Consequently, $|N_{G\circ H}\big{(}(x,h)\big{)}\cap B|\leq2$ for all $x\in V(G)\setminus P$ and $h\in V(H)$. In fact, we have proved that $B$ is a $2$LP set in $G\circ H$. Therefore, $L_{2}(G\circ H)\geq|B|=2\rho(G)$.  
\end{proof}

Mojdeh et al. \cite{msv} showed that $\rho(G)\leq n-\Delta(G)$ for all graphs $G$. Moreover, they proved that the family $\Gamma$, satisfying the following conditions, characterizes all graphs attaining this bound. $(i)$ There exists a vertex $v\in N(u)$, of a vertex $u$ of the maximum degree, such that $N[v]\subseteq N[u]$, $(ii)$ the subset $V(G)\setminus N[u]$ is independent, and $(iii)$ every vertex in $N[u]\setminus N[v]$ has at most one neighbor in $V(G)\setminus N[u]$. In fact, both lower and upper bounds in Theorem \ref{lex1} are sharp for all graphs $G$ in $\Gamma$ and every nontrivial graph $H$, simultaneously. Moreover, the bounds given in Theorem \ref{lex1} can be considered as an improvement of the bound given in \cite{msv}.

We end this section with bounding the $2$LP partition number of lexicographic product graphs.

\begin{theorem}\label{lex2}
Let $G$ and $H$ be two graphs such that $\delta(G)\geq 1$. Then, 
$$\lceil \frac{(\Delta(G)+1)|V(H)|}{2}\rceil\leq \chi_{\times2}(G\circ H)\leq \chi_{\times2}(G)|V(H)|.$$
Moreover, these bounds are sharp.
\end{theorem}
\begin{proof}
Let $v$ be a vertex of maximum degree in $G$. As shown in the proof of Theorem \ref{lex1}, every $2$LP set of $G\circ H$ has at most two vertices of $N_G[v]\times V(H)$. Therefore, every $2$LP partition of $G\circ H$ has at least $\lceil |N_G[v]\times V(H)|/2 \rceil$ members. Hence, $\chi_{\times2}(G)\geq \lceil(\Delta(G)+1)|V(H)|/2\rceil$.

We now prove the upper bound. Let $V(H)=\{u_1,\cdots,u_{n'}\}$ and let $\{B_1,\cdots,B_{\chi_{\times2}(G)}\}$ be a $2$LP partition of $G$. Consider the following sets
$$B_i\times \{u_1\},B_i\times \{u_2\},\cdots,B_i\times \{u_{n'}\}$$
for $1\leq i\leq \chi_{\times2}(G)$. It is clear that they give a partition of $V(G\circ H)$. Let $(v_1,u_k),(v_2,u_k),(v_3,u_k)\in N_{G\circ H}[(v,u_s)]\bigcap(B_i\times \{u_k\})$ for some $v\in V(G)$ and $1\leq s\leq n'$. This implies that $v_1,v_2,v_3 \in N_{G}[v]\cap B_i$, which contradicts the fact that $B_i$ is a $2$LP set in $G$. Therefore, $B_i\times \{u_j\}$ is a $2$LP set in $G\circ H$ for every $1\leq i\leq \chi_{\times2}(G)$ and $1\leq j\leq n'$. Hence $\chi_{\times 2}(G\circ H)\leq \chi_{\times 2}(G)|V(H)|$.

In view of Theorem \ref{Tree}, both lower and upper bounds are sharp for each tree with odd maximum degree and all graphs $H$.
\end{proof}


\section{On $2$-total limited packings}

We consider the problem of deciding whether a graph $G$ has a $2$TLP set of cardinality at least a given integer. That is stated in the following decision problem.

$$\begin{tabular}{|l|}
\hline
\mbox{$2$-TOTAL LIMITED PACKING problem ($2$TLP problem)}\\
\mbox{INSTANCE: A graph $G$ and a positive integer $k'$.}\\
\mbox{QUESTION: Is there a $2$TLP set in $G$ of cardinality at least $k'$?}\\
\hline
\end{tabular}$$

We make use of the following decision problem which is known to be NP-complete for bipartite graphs and for chordal graphs (see \cite{hs}).

$$\begin{tabular}{|l|}
\hline
\mbox{OPEN PACKING problem (OP problem)}\\
\mbox{INSTANCE: A graph $G$ of order $n$ and a positive integer $k$.}\\
\mbox{QUESTION: Is there an open packing in $G$ of cardinality at least $k$?}\\
\hline
\end{tabular}$$

\begin{theorem}	
The $2$TLP problem is NP-complete even for bipartite graphs and for chordal graphs.
\end{theorem}
\begin{proof}
The $2$TLP problem is a member of NP because checking that a set $S$ of vertices is a $2$TLP of cardinality at least $k$ can be done in polynomial time.

In what follows, we show how a polynomial time algorithm for the OP problem could be used to solve the $2$TLP problem in polynomial time. Let $G$ with $V(G)=\{v_{1},\cdots,v_{n}\}$ be a graph as an instance of the OP problem. We set $G'=G\odot K_1$, that is, a graph obtained from $G$ by joining a new vertex $u_{i}$ to $v_{i}$ for each $1\leq i\leq n$, and set $k'=n+k$. Let $B$ be a $\rho_{o}(G)$-set. It is readily seen by the construction that $B'=B\cup \{u_1,\cdots,u_n\}$ is a $2$TLP set of the graph $G'$. So, $L_{2,t}(G')\geq |B'|=\rho_{o}(G)+n$.

Suppose, conversely, that $B'$ be an $L_{2,t}(G')$-set. If we have $u_i\notin B'$ for some $1\leq i\leq n$ and $|N_{G}(v_i)\cap B'|\leq1$, then $B'\cup \{u_i\}$ is a $2$TLP set in $G'$, which contradicts the maximality of $B'$. Hence, if $u_i\notin B'$, then $|N_G(v_i)\cap B'|=2$. Let $v_j \in N_G(v_i)\cap B'$. Then, it can be easily checked that $B''=(B'\backslash\{v_j\})\cup \{u_i\}$ is an $L_{2,t}(G')$-set containing $u_i$. Therefore, without loss of generality, we may assume that $\{u_1,\cdots,u_n\}\subseteq B'$. With this in mind, it is now clear that $S=B'\setminus \{u_1,\cdots,u_n\}$ is an open packing in $G$. So, $|S|=|B'|-n\leq \rho_o(G)$ and we have that $L_{2,t}(G')\leq\rho_o(G)+n$.

We now get from $L_{2,t}(G')=\rho_o(G)+n$ that $L_{2,t}(G')\geq k'$ if and only if $\rho_o(G)\geq k$. It is known from \cite{hs} that the OP problem is NP-complete even for bipartite graphs and for chordal graphs. By the construction, if $G$ is a bipartite (resp. chordal) graph, then $G\odot k_1$ is also a bipartite (resp. chordal) graph. Consequently, the $2$TLP problem is NP-complete even for bipartite graphs and for chordal graphs.
\end{proof}  

From the result above, we conclude that the problem of computing the $2$TLP number is NP-hard, even for some special families of graphs. Consequently, it would be desirable to bound this parameter in terms of several invarients of graphs. Several bounds on the $2$TLP number were given in \cite{hms}. We here prove that this parameter of a graph $G$, except for a very especial case, can be bounded from below just in terms of the girth of $G$.

Let $\Lambda$ be the family of graphs in which every three vertices has a common neighbor. Evidently, $2\leq L_{2,t}(G)\leq|V(G)|$ holds for all nontrivial graphs $G$. Moreover, equality holds in the lower bound (resp. upper bound) if and only if $G\in \Lambda$ (resp. $\Delta(G)\leq2$).

\begin{theorem}\label{girth}
For any graph $G\notin \Lambda$ with a cycle, $L_{2,t}(G)\geq g(G)$. This bound is sharp. 
\end{theorem}
\begin{proof}
Since $G\notin \Lambda$, it follows that the lower bound holds for $g(G)=3$. So, we assume that $g(G)\geq4$. Let $C:v_{1}v_{2}\cdots v_{g(G)}v_{1}$ be a shortest cycle in $G$ and let $S=V(C)$. Clearly, every vertex in $S$ has exactly two neighbors in $S$. Suppose that there exists a vertex $u\notin S$ with three neighbors $v_{i},v_{j},v_{k}\in S$, in which $1\leq i<j<k\leq g(G)$. Because $g(G)\geq4$ and $v_{i},v_{j},v_{k}\in N(u)$, we have $j-i,k-j\geq2$. Without loss of generality, we can assume that $k-j\geq j-i$. Consider the cycle $C'$ with the vertices from the shortest $v_{i},v_{j}$-path on $C$ along with $u$. Since $C$ is a shortest cycle in $G$, we have $|V(C)|\leq|V(C')|=j-i+2$. So, $|V(C)|-(j-i)\leq2$. Since $k-j\geq2$, this necessarily implies that $v_{i}=v_{k}$, a contradiction. In fact, we have proved that every vertex of $G$ has at most two neighbors in $S$. Therefore, $L_{2,t}(G)\geq|S|=g(G)$.

The bound is sharp for all graphs $G$ obtained from a cycle $C:v_{1}v_{2}\cdots v_{|V(C)|}v_{1}$ by joining some vertex $v_{i}$ to a pendant vertex $u_{i}$ provided that $|i-j|\geq3$ for all such vertices $v_{i}$ and $v_{j}$. Moreover, all complete bipartite graphs $K_{m,n}$ for $m,n\geq2$, with $L_{2,t}(K_{m,n})=g(K_{m,n})=4$, attain the lower bound. 
\end{proof}

Let $B$ and $\overline{B}$ be an $L_{2,t}(G)$-set and an $L_{2,t}(\overline{G})$-set, respectively. If $|B\cap \overline{B}|\geq6$, then there exists a vertex $v\in B\cap \overline{B}$ such that $|N_{G}(v)\cap B|\geq 3$ or $|N_{\overline{G}}(v)\cap \overline{B}|\geq 3$, which is impossible. Therefore, $n\geq|B\cup \overline{B}|=|B|+|\overline{B}|-|B\cap \overline{B}|\geq|B|+|\overline{B}|-5$. This shows that $L_{2,t}(G)+L_{2,t}(\overline{G})\leq n+5$. In what follows, we prove that the cycle $C_5$ is the only graph for which the equality holds in this Nordhaus-Gaddum inequality. Note that it can be considered as a total version of the inequality (\ref{B}) concerning the $2$LP numbers, which was proved by induction on the order of graph. However, we give a non-inductive proof here.

\begin{theorem}
Let $G$ be a graph of order $n$ different from $C_5$. Then, $L_{2,t}(G)+L_{2,t}(\overline{G})\leq n+4$.
\end{theorem}
\begin{proof}  
We first give another proof for $L_{2,t}(G)+L_{2,t}(\overline{G})\leq n+5$. Let $u$ be a vertex of maximum degree in $G$ and $B$ be an $L_{2,t}(G)$-set. Then, $|B\cap N_G[u]|\leq 3$.
On the other hand,
$$|B|=|B\cap N_G[u]|+|B\cap (V(G)\backslash N_G[u])|\leq 3+|V(G)\backslash N_G[u]|=3+n-(\Delta(G)+1)=n-\Delta(G)+2.$$
Hence, $L_{2,t}(G)\leq n-\Delta(G)+2$. By considering the analogous inequality for $\overline{G}$, we have
$$L_{2,t}(G)+L_{2,t}(\overline{G})\leq n-\Delta(G)+2+n-\Delta(\overline G)+2=n+5+\delta (G)-\Delta (G) \leq n+5.$$

Let $G$ be a graph for which $L_{2,t}(G)+L_{2,t}(\overline{G})=n+5$. With this in mind, and taking the latter inequality chain into account, we necessarily get\\
$(i)$ $\delta (G)-\Delta(G)=0$ and hence $G$ is regular,\\
$(ii)$ $L_{2,t}(G)=n-\Delta(G)+2$, and\\
$(iii)$ $L_{2,t}(\overline{G})=\delta(G)+3=\Delta(G)+3$.

Let $P$ be an $L_{2,t}(G)$-set. We have $|V(G)\setminus P|=\Delta(G)-2$ as $|P|=n-\Delta(G)+2$. Therefore, $\Delta(G)=|V(G)\setminus P|+2$. In what follows, we prove that $P=V(G)$. Suppose to the contrary that there is a vertex $x\in V(G)\setminus P$. In particular, $\deg_{G}(x)=|V(G)\setminus P|+2$. Because $x$ has at most two neighbors in $P$, it follows that $x$ has at least $|V(G)\setminus P|$ neighbors outside of $P$, which contradicts the fact that $x\in V(G)\setminus P$. This shows that $|V(G)\setminus P|=0$, and hence $G$ is an $r$-regular graph with $r\in \{1,2\}$. Note that $r\neq1$ since $L_{2,t}(G)=n+1$ (followed from $(ii)$) does not make sense. Thus, $G$ is a union of cycles. Now we prove that $G$ contains only one cycle. Assume for the sake of contradiction that $G$ contains $k\geq2$ cycles $C_{1},C_{2},\cdots,C_{k}$ of orders $n_{1},n_{2},\cdots,n_{k}$, respectively ($n_1+n_2+...+n_k=n$).

Let $\overline{P}$ be an $L_{2,t}(\overline{G})$-set. We observe that $\overline{G}$ contains a subgraph $K_{n_1, n_2, ..., n_k}$ with partite sets $A_{1},A_{2},\cdots,A_{k}$ of cardinalities $n_1,n_2,\cdots,n_{k}$, respectively. Since $\overline{P}$ is a $2$TLP set in $\overline{G}$, it follows that $|\overline{P} \cap A_{i}|\leq 2$ for each $1\leq i\leq k$. If $k=2$, then $L_{2,t}(\overline{G})=\delta(G)+3\leq4$, which is a contradiction. Therefore, $k\geq3$. Since $L_{2,t} (\overline{G})=5$, the distribution of five vertices in the partite sets $A_{1},A_{2},\cdots,A_{k}$ is of the form:\\
$\bullet$ two vertices in $A_{i_1}$, two vertices in $A_{i_2}$ and one vertex in $A_{i_3}$ for some $1\leq i_{1}, i_{2},i_{3}\leq k$; or\\
$\bullet$ two vertices in $A_{i_1}$, and one vertex in each $A_{i_2},A_{i_3},A_{i_4}$ for some $1\leq i_{1},i_{2},i_{3},i_{4}\leq k$; or\\
$\bullet$ one vertex in each partite sets $A_{i_1},\cdots,A_{i_5}$ for some $1\leq i_{1},\cdots,i_{5}\leq k$.\\
In all three cases above, any vertex from $A_{i_1}$ is adjacent to at least three vertices in $P\cap(\cup_{j\neq i_1} A_{j})$, which is a contradiction. Consequently, $k=1$. This shows that $G$ is isomorphic to $C_n$, and hence $L_{2,t}(G)=n$.

We now prove that the order of $G$ is not greater than five. Suppose to the contrary that $n\geq 6$. Since $\overline{G}$ is $(n-3)$-regular (with $n-3\geq3$), there exists a vertex $x\in V(\overline{G})\setminus \overline{P}$. Together $|\overline{P}|=5$, $|N_{\overline{G}}[x]|=n-2$ and that $\overline{P}$ in a $2$TLP set of $\overline{G}$ imply that $|\overline{P}\cap N_{\overline {G}}[x]|=3$ and $|\overline{P}\cap(V(\overline {G})\backslash N_{\overline{G}}[x])|=2$.
In particular, $x$ must be in $\overline{P}$ (for otherwise $|\overline{P}\cap N_{\overline {G}}(x)|=3>2$, which is impossible). This contradiction shows that $n\leq5$. On the other hand, it is readily seen that $L_{2,t}(G)+L_{2,t}(\overline{G})<n+5$ when $G\in\{C_{3},C_{4}\}$. Therefore, $G\cong C_5$.
\end{proof}

It is clear from the definitions that $L_{2,t}(G)\geq L_{2}(G)$ for all graphs $G$. In what follows, we prove that the difference between these two parameters for each tree can be bounded from above by one-third its order. Note that this does not hold in general. In fact, $L_{2,t}(G)-L_{2}(G)-|V(G)|/3$ can be arbitrarily large. To see this, consider the path $P:u_{1}u_{2}\cdots u_{3p}$ with $p\geq1$. For each $1\leq i\leq3p$, we add two $4$-cycles $C_{4,x}^{i}:x_{i1}x_{i2}x_{i3}x_{i4}x_{i1}$ and $C_{4,y}^{i}:y_{i1}y_{i2}y_{i3}y_{i4}y_{i1}$ and join the vertex $u_{i}$ to $x_{i1}$ and $y_{i1}$. Let $G$ be the resulting graph. It is then easy to see that
\begin{equation*}
S=\cup_{i=1}^{3p}\big{(}V(C_{4,x}^{i})\cup V(C_{4,y}^{i})\big{)} \ \mbox{and} \ S'=\{u_{3i-2},u_{3i-1}\}_{i=1}^{p}\cup\big{(}\cup_{i=1}^{3p}\{x_{i2},x_{i3},y_{i2},y_{i3}\}\big{)}
\end{equation*}
are an $L_{2,t}(G)$-set and an $L_{2}(G)$-set, respectively. Then, $L_{2,t}(G)-L_{2}(G)=|S|-|S'|=10p=|V(G)|/3+p$. 

In order to prove the above-mentioned result concerning trees, we need the following simple but useful lemma.

\begin{lemma}\label{leaves}
Given any tree $T$ on at least two vertices, there exists an $L_{2}(T)$-set \emph{(}$L_{2,t}(T)$-set\emph{)} containing the leaf adjacent to every weak support vertex and two leaves adjacent to every strong support vertex of $T$.
\end{lemma}

\begin{theorem}
If $T$ is a tree of order $n$, then $L_{2,t}(T)-L_2(T)\leq \lfloor\frac{n}{3}\rfloor$. Moreover, this bound is sharp.
\end{theorem}
\begin{proof}
We prove the upper bound by induction on the order $n$ of $T$. It can be readily checked that if $T$ is a tree of diameter at most three, then the result is true. So, in what follows, we may assume that diam$(T)\geq4$, which implies that $n\geq5$.

Suppose that the inequality holds for all trees $T'$ of order $n'\geq1$, and let $T$ be a tree of order $n>n'$. Let $P$ be a diametral $r,c$-path in $T$ and root $T$ at $r$. Clearly, both $r$ and $c$ are leaves in $T$. Let $d$ be the parent of $c$.

Note that, by the definitions, at most two leaves of every support vertex of $T$ belong to each $2$LP set ($2$TLP set) in $T$. This implies that if $T$ contains a support vertex adjacent to at least three leaves, then removing one of the leaves leads to a tree $T'$ of order $n'=n-1$ with $L_{2,t}(T')=L_{2,t}(T)$ and $L_2(T')=L_2(T)$. So, by the induction hypothesis, we have 
$$L_{2,t}(T)-L_2(T)=L_{2,t}(T')-L_2(T')\leq \frac{n'}{3}<\frac{n}{3}.$$
Thus, from now on, we assume that every support vertex of $T$ is adjacent to at most two leaves. In particular, $d$ has degree at most three in $T$. Let $e$ be the parent of $d$ in $T$. We distinguish two cases depending of $\deg(d)$.\vspace{1mm}\\
\textit{Case 1}. $\deg(d)=2$. We need to consider two more possibilities depending on the behavior of $e$.\vspace{0.5mm}

\textit{Subcase 1.1}. $\deg(e)=2$. Let $f$ be the parent of $e$ in $T$ and consider the tree $T'=T-\{c,d,e\}$ of order $n'=n-3\geq2$ (note that $f\neq r$ since diam$(T)\geq4$). Let $S$ be an $L_{2,t}(T)$-set. If $e\notin S$, it follows that $|S\cap N_{T}(f)|=2$ (for otherwise, $S\cup\{e\}$ is a $2$TLP set in $T$, which is impossible). Let $x\in S\cap N_{T}(f)$. It is then easy to see that $(S\setminus\{x\})\cup\{e\}$ is an $L_{2,t}(T)$-set, as well. Therefore, we may assume that $e\in S$. Moreover, both $c$ and $d$ necessarily belong to $S$. We now observe that $S\setminus\{c,d,e\}$ is a $2$TLP of $T'$, and so, $L_{2,t}(T')\geq L_{2,t}(T)-3$. On the other hand, every $L_2(T')$-set can be extended to a $2$LP of $T$ by adding $d$ and $c$ to it. Hence, $L_2(T)\geq L_2(T')+2$. Thus, by the induction hypothesis, we have 
$$L_{2,t}(T)-L_2(T)\leq L_{2,t}(T')+3-L_2(T')-2\leq \frac{n'}{3}+1=\frac{n}{3}.$$

\textit{Subcase 1.2}. $\deg(e)=k+1\geq3$. Let $e$ be adjacent to a leaf $g$. Let $T'=T-\{c\}$. Notice that every $L_{2,t}(T)$-set contains the vertex $c$. So, $L_{2,t}(T')\geq L_{2,t}(T)-1$. On the other hand, there exists an $L_2(T')$-set $S'$ containing both $g$ and $d$ by Lemma \ref{leaves}. This shows that $e\notin S'$. Thus, $S'\cup \{c\}$ is a $2$LP of $T$. This implies that $L_2(T)\geq L_2(T')+1$. We now get
$$L_{2,t}(T)-L_2(T)\leq L_{2,t}(T')+1-L_2(T')-1\leq \frac{n'}{3}<\frac{n}{3}.$$

Hence, we assume that every child of $e$ has degree at least two. Suppose that $d_1,d_2,\cdots,d_k$ are the children of $e$ and $c_i$ is a leaf adjacent to $d_i$, $1\leq i\leq k$, where $d_1=d$ and $c_1=c$. We first suppose that $k\geq3$ and consider the tree $T'=T-\{d_1, c_1\}$ of order $n'=n-2$. Every $L_{2,t}(T)$-set $S$ contains at most two children of $e$. So we can assume, without loss of generality, that $d_{1}\notin S$. It then follows that $c_1\in S$ by the maximality of $S$. Hence, $S\setminus \{c_1\}$ is a $2$TLP of $T'$. So, $L_{2,t}(T')\geq L_{2,t}(T)-1$. On the other hand, $S'\cup\{c_{1}\}$ is a $2$LP set of $T$ for any $L_2(T')$-set $S'$. Thus, $L_2(T)\geq|S'\cup \{c_1\}|=L_2(T')+1$. Therefore, 
$$L_{2,t}(T)-L_2(T)\leq L_{2,t}(T')+1-L_2(T')-1\leq \frac{n'}{3}<\frac{n}{3}.$$

Suppose now that $k=2$. By Lemma \ref{leaves}, there exists an $L_{2,t}(T)$-set $S$ containing all leaves adjacent to $d_1$ and $d_2$. We distinguish two more possibilities.\vspace{0.5mm}

\textit{Subcase 1.2.1}. The vertex $d_2$ has two leaves $c_2$ and $c_{2}'$. We set $T'=T-\{c_1,c_2,c_{2}'\}$. Note that $S\setminus \{c_1,c_2,c_{2}'\}$ is a $2$TLP set in $T'$. Thus, $L_{2,t}(T')\geq L_{2,t}(T)-3$. Furthermore, there exists an $L_2(T')$-set $S'$ containing $d_1$ and $d_2$ by Lemma \ref{leaves}. Hence, $S'\cup \{c_1,c_2\}$ is a $2$LP set in $T$, and we have $L_2(T)\geq L_2(T')+2$. Now, by the induction hypothesis, we get 
$$L_{2,t}(T)-L_2(T)\leq L_{2,t}(T')+3-L_2(T')-2\leq \frac{n'}{3}+1=\frac{n}{3}.$$

\textit{Subcase 1.2.2}. The vertex $d_2$ has only the leaf $c_2$. We set $T'=T-\{c_1,c_2\}$. Then, $S\setminus \{c_1,c_2\}$ is a $2$TLP set in $T'$. Hence, $L_{2,t}(T')\geq L_{2,t}(T)-2$. Similarly to the discussion given in Subcase 1.2.1, we have $L_2(T)\geq L_2(T')+2$. Therefore, 
$$L_{2,t}(T)-L_2(T)\leq L_{2,t}(T')+2-L_2(T')-2\leq \frac{n'}{3}<\frac{n}{3}.$$
\textit{Case 2}. $\deg(d)=3$. Let $d$ have two leaves $c_1=c$ and $c_2$. Suppose first that $e$ is adjacent to precisely one leaf $g_1$ and consider the tree $T'=T-\{g_1\}$. By Lemma \ref{leaves}, there exists an $L_{2,t}(T)$-set $S$ containing $g_1$. Then, $L_{2,t}(T')\geq|S\setminus \{g_1\}|=L_{2,t}(T)-1$ as $S\setminus \{g_1\}$ is a $2$TLP in $T'$. On the other hand, there is an $L_2(T')$-set $S'$ such that $N_{T'}[d]\cap S'=\{c_1,c_2\}$ by Lemma \ref{leaves}. This necessarily implies that $d,e\notin S'$. In such a situation, $S'\cup \{g_1\}$ is a $2$LP in $T$. Hence, $L_2(T)\geq L_2(T')+1$. Therefore, 
$$L_{2,t}(T)-L_2(T)\leq L_{2,t}(T')+1-L_2(T')-1\leq \frac{n'}{3}<\frac{n}{3}.$$

Suppose now that $e$ is adjacent to (exactly) two leaves $g_1$ and $g_2$. We set $T'=T-\{g_2\}$. Note that every $L_{2,t}(T)$-set $S$ contains two vertices in $N_{T}(e)$. So, we can assume that $d,g_1\in S$ and that the other children of $e$ do not belong to $S$. This in particular implies that $L_{2,t}(T')\geq L_{2,t}(T)$. In addition, $L_2(T)\geq L_2(T')$ since $T$ is obtained from $T'$ by joining a leaf to $e$. We now get, 
$$L_{2,t}(T)-L_2(T)\leq L_{2,t}(T')-L_2(T')\leq \frac{n'}{3}<\frac{n}{3}.$$

Hence, we may assume that $e$ does not have any leaf. If $e$ has a child $d'$ of degree two, then we get the desired inequality similar to Case 1. So, let every child of $e$ have degree three. Let $e$ have $k'$ children $d_1,d_2,\cdots,d_{k'}$, and let $d_i$ have two leaves $c_{i1}$ and $c_{i2}$ for each $1\leq i\leq k'$. Consider the tree $T'=T-\{e,d_1,\cdots,d_{k'},c_{11},c_{12},\cdots,c_{k'1},c_{k'2}\}$ of order $n'=n-3k'-1\geq2$. We need to consider two possible cases depending of $k'$.\vspace{0.5mm}

\textit{Subcase 2.1}. $k'\geq2$. There is an $L_{2,t}(T)$-set $S$ such that $\{c_{11},c_{12},\cdots,c_{k'1},c_{k'2}\}\subseteq S$ by Lemma \ref{leaves}. Moreover, we may assume that $d=d_1,d_2\in S$ by the maximality of $S$. In such a situation, $S\setminus \{d_1,d_2,c_{11},c_{12},\cdots,c_{k'1},c_{k'2}\}$ is a $2$TLP set in $T'$. Therefore, $L_{2,t}(T')\geq L_{2,t}(T)-2k'-2$. On the other hand, $S'\cup \{c_{11},c_{12},\cdots,c_{k'1},c_{k'2}\}$ is a $2$LP in $T$ for each $L_2(T')$-set $S'$. Hence, $L_2(T)\geq L_2(T')+2k'$. Together the last two inequalities and the induction hypothesis, we deduce that
$$L_{2,t}(T)-L_2(T)\leq L_{2,t}(T')+2k'+2-L_2(T')-2k'\leq \frac{n'}{3}+2<\frac{n}{3}.$$

\textit{Subcase 2.2}. $k'=1$. We set $T'=T-\{e,d_{1},c_{11},c_{12}\}$. A similar discussion shows that $\{d_{1},c_{11},c_{12}\}\subseteq S$ for some $L_{2,t}(T)$-set $S$, and that $S\setminus \{d_{1},c_{11},c_{12}\}$ is a $2$TLP set in $T'$. This leads to $L_{2,t}(T')\geq L_{2,t}(T)-3$. Moreover, any $L_2(T')$-set $S'$ can be extended to a $2$LP in $T$ by adding $c_{11}$ and $c_{12}$ to it. Therefore, $L_2(T)\geq L_2(T')+2$. We now have
$$L_{2,t}(T)-L_2(T)\leq L_{2,t}(T')+3-L_2(T')-2\leq \frac{n'}{3}+1<\frac{n}{3}.$$
Indeed, in all possible cases, we have proved the desired upper bound.

That the bound is sharp, may be seen as follows. Consider three paths $P_{1}:x_{1}\cdots x_{3t-1}$, $P_{2}:y_{1}\cdots y_{3t}$ and $P_{3}:z_{1}\cdots z_{3t-1}$ for any positive integer $t$. Let $T$ be obtained from $P_{1}+P_{2}+P_{3}$ by joining a new vertex $u$ to the vertices $x_1$, $y_1$ and $z_1$. We observe that $V(T)\setminus\{x_1\}$ is an $L_{2,t}(T)$-set of cardinality $9t-2$. Moreover, it is easy to see that $\{x_{3i-1},x_{3i}\}_{i=1}^{t-1}\cup \{y_{3i-1},y_{3i}\}_{i=1}^{t}\cup \{z_{3i},z_{3i+1}\}_{i=1}^{t-1}\cup \{u,x_{3t-1},z_{1}\}$ is an $L_{2}(T)$-set of cardinality $6t-1$. Therefore, $L_{2,t}(T)-L_{2}(T)=3t-1=\lfloor|V(T)|/3\rfloor$. This completes the proof.
\end{proof}


\section{Concluding remarks}

Concerning our contributions to the total version of limited packing in graphs in Section $4$, a natural problem arises here. In fact, the problem of vertex partitioning of graphs into $k$TLP sets can be considered from some points of view. It is the dual problem of vertex partitioning into $k$-tuple total dominating sets (see \cite{cdh,sv} for more information). Moreover, it is a generalization of the concept of injective coloring in graphs as the color class of this graph coloring are the very open packing sets (see \cite{bsy,hkss}).

We bounded $\chi_{\times2}$ for the lexicographic product graphs in this paper. It could be of interest to give the exact value of this parameter or bounding it for the other standard product graphs (the Cartesian, direct and strong ones). We conclude with the following problem.\vspace{0.5mm}\\
\textbf{Problem.} Determine the complexity of the decision problem associated with $\chi_{\times k}$ for $k\geq2$. In fact, it is worth proving that the problem is NP-complete (in particular, for some important families of graphs).\vspace{4.5mm}\\
\textbf{Acknowledgments}\vspace{0.5mm}\\
This work has been supported by the Discrete Mathematics Laboratory of the Faculty of Mathematical Sciences at Alzahra University.


\end{document}